\documentclass[10pt,a4paper,reqno]{amsart}
\usepackage{amsmath, amsthm, amsfonts, amssymb, mathrsfs}
\usepackage{graphicx,xcolor}
\usepackage{subcaption}
\usepackage{mathtools}
\usepackage{enumerate}
\usepackage{eucal}
\usepackage{enumitem}
\usepackage{marginnote}
\usepackage[export]{adjustbox}
\usepackage{url}
\newtheorem{theorem}{Theorem}[section]
\newtheorem{lemma}[theorem]{Lemma}
\newtheorem{proposition}[theorem]{Proposition}

\theoremstyle{remark}
\newtheorem{remark}[theorem]{Remark}
\theoremstyle{definition}

\newtheorem{ass}{Assumption}

\numberwithin{equation}{section}

\newcommand{\Dom}{{\operatorname{Dom}}}
\newcommand{\Ker}{{\operatorname{Ker}}}

\newcommand{\nl}{{\operatorname{nul}}}
\newcommand{\df}{{\operatorname{def}}}

\newcommand\la{\lambda}

\newcommand\wt{\widetilde}
\newcommand\wh{\widehat}

\newcommand\cD{{\mathcal D}}

\newcommand\sS{{\mathscr S}}
\newcommand\sC{{\mathscr C}}
\newcommand{\NN}{{\mathbb N}}
\newcommand{\RR}{{\mathbb R}}
\newcommand{\CC}{{\mathbb C}}

\renewcommand\d{{\rm d}}

\newcommand{\ii}{{\rm i}}
\renewcommand\o{{\rm o}}

\newcommand\ds{\displaystyle}
\newcommand{\ov}[1]{\overline{#1}}
\newcommand\mydot{\,\cdot\,}

\newcommand\sess{\sigma_{\rm ess}}

\DeclareMathOperator\Ran{Ran}

\DeclareMathOperator\Real{Re}

\renewcommand\Re{\Real}

{\large
\begin{document}

\title[Essential spectrum of mixed-order systems of $\Psi$DOs]
{Essential spectrum of elliptic systems of pseudo-differential operators on $L^2(\RR^N) \oplus  L^2(\RR^N)$}

\subjclass[2010]{47G30, 35S05, 47A10, 47A53}
\keywords{Essential spectrum, pseudo-differential operator, mixed-order system, Douglis-Nirenberg ellipticity, Schur complement, approximate inverse}
\date{\today}

\author{Orif \,O.\ Ibrogimov and Christiane Tretter}
\address[O.\,O.\ Ibrogimov]{%
	Department of Mathematics, 
	University College London,
	Gower Street, London,
	WC1E 6BT, UK}
\email{o.ibrogimov@ucl.ac.uk}

\address[Christiane Tretter]{%
	Mathematisches Institut, 
	Universit\"{a}t Bern,
	Sidlerstrasse.\ 5,
	3012 Bern, Switzerland}
\email{tretter@math.unibe.ch}

\date{\today}


\begin{abstract}
Inspired by a result of M.\ W.\ Wong \cite{Wong88-Communications}, we establish an analytic description of the essential spectrum of non-self-adjoint mixed-order systems of pseudo-differential operators on $L^2(\RR^N) \oplus  L^2(\RR^N)$ that are uniformly Douglis-Nirenberg elliptic with positive-order diagonal entries. We apply our result to a problem arising in the dynamics of falling liquid \vspace{-2mm} films.
\end{abstract}

\maketitle


\section{Introduction}
The aim of  this paper is to establish an analytic description of the essential spectrum of operator matrices acting in the Hilbert space $L^2(\RR^N) \oplus L^2(\RR^N)$
induced by non-self-adjoint $2\times 2$ matrix  pseudo-differential operators
	\begin{equation}\label{T0}
	\begin{aligned} 
	T_0 & := \begin{pmatrix} T_a & T_b \\ T_c & T_d \end{pmatrix}, \quad \Dom(T_0) :=  \sS(\RR^N) \oplus \sS(\RR^N).
	\end{aligned}
	\end{equation}
Here $T_a$, $T_b$, $T_c$, $T_d$ are pseudo-differential operators of mixed orders $m$, $n$, $p$, $q\in \RR$ on the Schwartz space $\sS(\RR^N)$ with classical symbols. 
We require that $T_0$ is uniformly Douglis-Nirenberg elliptic on $\RR^N$ and that both $T_a$ and $T_d$ have positive orders,
without loss of generality, we may assume that the order of $T_a$ is greater than or equal to that of $T_d$, \vspace{-2mm} i.e.
\begin{align} 
\label{orders}
m\geq q>0.
\end{align} 

The spectral analysis of such operators plays a crucial role in stability problems of several branches of theoretical physics, in particular, in the dynamics of fluids and magnetism, see e.g.\ \cite{Landau-Lifshitz-87b}, \cite{Lifschitz-89b}. For example, the linear operators arising in the dynamics of Ekman flow, Hagen-Poiseuille flow or a liquid film falling down a vertical wall are exactly of this type, see e.g.\ \cite{Greenberg-marletta-04}, \cite{MT07}, \cite{MT13}, \cite{Film-book-2012}, \cite{Pradas-Tseluiko-Kalliadasis-2011}.

In compact subdomains of $\RR^N$ the spectral properties of matrix (ordinary, partial, pseudo-) differential operators of this type were extensively studied during the last forty years. The most general results on the essential spectrum of $\ell\times\ell$ mixed-order systems of partial differential operators of Douglis-Nirenberg type on compact $N$-dimensional manifolds with boundary were established by G.\ Grubb and G.\ Geymonat \cite{GG77} for a special case of orders which, for $\ell=2$ as above, amounts to  $m=2n=2p > q = 0$.
%
In their celebrated work \cite{ALMS94}, F.V.\ Atkinson, H.\ Langer, R.\ Mennicken and A.A.\ Shkalikov obtained results on the essential spectra of $2\times2$ operator matrices in an abstract setting. As an application, they established an analytic description of the essential spectrum of mixed-order systems of \emph{ordinary} differential operators over compact intervals. 

In non-compact subdomains of $\RR^N$ not much is known in the general setting, a partial exception being the case when the operator matrix is \emph{symmetric} with \emph{ordinary} differential operator entries of specific orders, see \cite{IST15}, \cite{ILLT13} and the references therein. 
In the non-self-adjoint setting, only recently a first step was made in \cite{Ibrogimov-15} for \emph{ordinary} matrix differential operators on $\RR$ under certain restrictions on the coefficients and on the orders of the matrix entries. There the second diagonal entry $D$ is assumed to have \emph{zero order}, while the order of $A$ is required to be the sum of the orders of the off-diagonal entries, i.e.\ $q=0$ and $m=n+p$. 

An analysis of the essential spectrum of (the closures of) matrix pseudo-differen\-tial operators $T_0$ in $L^2(\RR^N)\oplus L^2(\RR^N)$ as in \eqref{T0} seems to be lacking so far not only for the case that $T_d$ has zero order, but also when $T_d$ has positive order, i.e.~$q>0$. In the latter case the essential spectrum has no local origin and may only be caused by the singularity at infinity. While a fairly complete Fredholm theory for Douglis-Nirenberg elliptic systems on $\RR^N$ is available when the latter are considered as \emph{bounded operators} on the scale of classical Sobolev spaces, see e.g.\ \cite{Rabier2012}, the spectral analysis in $L^2(\RR^N)\oplus L^2(\RR^N)$ requires us to consider them as unbounded~operators. 

In the current manuscript, we study the essential spectrum of the closure $\ov{T}_0$ of $T_0$ in $L^2(\RR^N)\oplus L^2(\RR^N)$ for the case when \emph{both} diagonal entries have \emph{positive orders}. Our aim is to provide an analytic description of the essential spectrum based on its origin in the singularity at infinity. This result is obtained in two steps. First we characterize the essential spectrum of the matrix pseudo-differential operator in terms of the essential spectrum of the so-called first Schur complement; here the main tools are approximate and generalized inverses of (semi-) Fredholm operators. Then, assuming that the symbol of the Schur compliment is in the Grushin symbol class and stabilizes at infinity in an appropriate sense, we apply a key result of M.W.\ Wong from \cite{Wong88-Communications}. Our characterization of the essential spectrum is explicit up to a certain exceptional set which is due to the use of the first Schur complement and which needs to be studied separately, e.g.\ using the closedness of the essential spectrum or by means of the second Schur complement. We illustrate our results by applying them to the falling liquid film problem considered in \cite{Film-book-2012}, \cite{Pradas-Tseluiko-Kalliadasis-2011} merely from a physical and numerical point of view.

The paper is organized as follows. Section 2 contains the operator-theoretic setting for the matrix pseudo-differential operator \eqref{T0} and its associated first Schur complement, the main working hypotheses and a discussion of the uniform ellipticity of the Schur complement. Section 3 provides all auxiliary results including the relationship between a parametrix and the generalized inverse of the Schur complement.
Section 4 is devoted to the main results of the paper. It contains the constructions of left approximate inverses for the operator matrix and its first Schur complement, see Theorems~\ref{thm:approx.inv} and \ref{thm2:approx.inv}. In Theorem~\ref{thm:main} the analytic description of the essential spectrum is given in terms of the limiting symbol of the Schur complement. Section 5 contains the application of our results to the falling liquid film problem.

The following notation is used throughout the paper. We write $\langle\mydot,\mydot\rangle$ for the inner product in $L^2(\RR^N)$. By $\sS(\RR^N)$ and $H^s(\RR^N)$, $s\in\RR$, respectively, we denote the Schwartz space and the 
$L^2$-Sobolev space of order $s$. For a Banach space $X$, we denote by $\sC(X)$ the set of closed linear operators acting in $X$. For $T\in\sC(X)$, we denote by $\Dom(T)$, $\Ker(T)$ and $\Ran(T)$ the domain, kernel and range of $T$, respectively. For a densely defined operator $T\in\mathscr{C}(X)$, $\sigma(T)$ denotes its spectrum. Further, $T\in\sC(X)$ is said to be \textit{Fredholm} if $\Ran(T)$ is closed and both $\nl(T):=\dim\Ker(T)$ and $\df(T):=\dim X/\Ran(T)$ are finite. For the essential spectrum, we use the definition 
	\[
    \sess (T) := \{ \la \in \CC : T-\la \text{ is not Fredholm}\},
	\]
which is the set $\sigma_{\rm{e}3}(T)$ in \cite[Section~IX.1]{EE87}. We use the notation $\langle x\rangle:=(1+|x|^2)^{1/2}$ for $x\in\RR^N$. For two functions $f:\Omega_1\to\RR$, $g:\Omega_2\to\RR$, where $\Omega_1$ and $\Omega_2$ are subsets of (not necessarily the same) Euclidean spaces, we write $f\lesssim g$ if there exists a universal constant $c>0$ such that $f(x_1)\leq cg(x_2)$ for $x_1\in\Omega_1$, $x_2\in\Omega_2$.

We shall need the following background from the theory of pseudo-differential operators, see e.g.\ \cite{Taylor-81b}, \cite{Wong-14b}. For $k\in\RR$, the H\"ormander symbol class $S^{k} (\RR^N\!\times\RR^N)$ $= S_{1,0}^{k} (\RR^N\!\times\RR^N)$ is defined to be the set of all infinitely smooth functions $\sigma\in C^\infty(\RR^N\!\times\RR^N)$ such that, for any two multi-indices $\alpha,\beta\in\NN_0^N$, there is a positive constant $C_{\alpha,\beta}$, 
depending only on $\alpha,\beta$, for which 
	\[
	|(\partial^{\beta}_x\partial^{\alpha}_{\xi})\sigma(x,\xi)| \leq C_{\alpha,\beta}\langle \xi \rangle^{k -|\alpha|}, \quad (x,\xi)\in\RR^N\!\times\RR^N.
	\]
Further, we set  
	\[
	S^{-\infty}(\RR^N\!\times\RR^N) := \bigcap_{k\in\RR}S^{k} (\RR^N\!\times\RR^N), \quad S^{\infty}(\RR^N\!\times\RR^N) := \bigcup_{k\in\RR}S^{k} (\RR^N\!\times\RR^N),
	\]
and we recall that for $\sigma\in S^{\infty}(\RR^N\!\times\RR^N)$, the pseudo-differential operator $T_{\sigma}$ with symbol $\sigma$ on the Schwartz space $\mathscr{S}(\RR^N)$ is 
defined by
	\begin{equation*}
	\ds(T_{\sigma}\phi)(x) := \frac{1}{(2\pi)^{\frac{N}{2}}}\int_{\RR^N}{\rm{e}}^{\ii x\cdot\xi}\sigma(x,\xi)\,\wh{\phi}(\xi)\,{\rm{d}}\xi, \quad \phi\in\Dom(T_{\sigma})=\mathscr{S}(\RR^N),
	\end{equation*}
where $\wh{\phi}$ is the Fourier transform of $\phi\in\mathscr{S}(\RR^N)$,
	\begin{equation*}
	\ds\wh{\phi}(\xi) := \frac{1}{(2\pi)^{\frac{N}{2}}}\int_{\RR^N}{\rm{e}}^{-\ii x\cdot\xi}\phi(x) \,{\rm{d}}x, \quad \xi\in\RR^N.
	\end{equation*}

By $\Psi^{k} (\RR^N\!\times\!\RR^N)$, $k\in\RR\cup\{\pm\infty\}$, we denote the set of pseudo-differential operators with symbols in $S^{k} (\RR^N\!\times\RR^N)$.

A pseudo-differential operator $T_\sigma\in\Psi^{k} (\RR^N\!\times\RR^N)$ is called \emph{uniformly elliptic} if its symbol $\sigma$ satisfies 
the relation                       
	\[
	|\sigma(x,\xi)| \gtrsim \langle\xi\rangle^{k} , \qquad x\in\RR^N, \ |\xi|\gtrsim 1.
	\]

A symbol $\sigma\in S^{k} (\RR^N\!\times\RR^N)$, $(x,\xi) \mapsto \sigma(x,\xi)$, is called \emph{homogeneous} (in $\xi$) of degree $k$ if for all $x\in\RR^N$ we have
	\[
	 \sigma(x,\tau\xi) = \tau^{k} \sigma(x,\xi), \qquad \tau>0, \ |\xi| \gtrsim 1.
	\]
By $S^{k} _h(\RR^N\!\times\RR^N)$ we denote the set of all homogeneous symbols of order $k$.

A symbol $\sigma\in S^{k} (\RR^N\!\times\RR^N)$ is called \textit{classical symbol of order $k$}, and we write $\sigma\in S_{cl}^{k} (\RR^N\!\times \RR^N)$, if $\sigma$ admits an 
asymptotic expansion 
	\begin{align}\label{symb:poly-2}
	\sigma(x,\xi) \sim \sum_{j=0}^\infty \sigma_{k-j}(x,\xi), \qquad x\in\RR^N, \ |\xi|\gtrsim 1,
	\end{align}
where $\sigma_{k-j} \in S_h^{k-j}(\RR^N\times\RR^N)$ is homogeneous of order $k-j$ for every $j\in\NN_0$; here \eqref{symb:poly-2} means that, \vspace{-2mm} for all $N\geq0$, 
	\[
	\sigma - \sum_{j=0}^N\sigma_j \in S^{k-N-1}(\RR^N\!\times\RR^N). 
	\]

The space of pseudo-differential operators with classical symbols of order $k$ is denoted by $\Psi_{cl}^{k} (\RR^N\!\times\RR^N)$. Note that for a classical symbol $\sigma\in S_{cl}^{k} (\RR^N\!\times \RR^N)$ the \emph{principal symbol} is the first term $\sigma_{k}$ in its asymptotic expansion \eqref{symb:poly-2}.

\section{Matrix pseudo-differential operators and associated \\first Schur complement}\label{sec:A.def}

	
In the Hilbert space $L^2(\RR^N) \oplus L^2(\RR^N)$, we consider the operator matrix $T_0$ given by \eqref{T0} where
$T_a\in\Psi^{m}_{cl}(\RR^N\!\times\RR^N)$, $T_b\in\Psi^{n}_{cl}(\RR^N\!\times\RR^N)$, 
$T_c\in\Psi^{p}_{cl}(\RR^N\!\times\RR^N)$, $T_d\in\Psi^{q}_{cl}(\RR^N\!\times\RR^N)$
are pseudo-differential operators in $L^2(\RR^N)$ defined on the Schwartz space $\sS(\RR^N)$ 
of orders $m$, $n$, $p$, $q \,\in\RR$ such that \eqref{orders} holds, i.e.\ $m\geq q>0$.

A simple integration by parts argument shows that $\Dom(T_0^*)$ is dense in $L^2(\RR^N)\oplus L^2(\RR^N)$ since it contains  $\sS(\RR^N) \oplus \sS(\RR^N)$. Therefore, $T_0$ is closable and we will denote the closure of $T_0$ by $T$.

We denote the symbols of the operators $T_a$, $T_b$, $T_c$, $T_d$ by $a$, $b$, $c$, $d$, and their principal symbols by $a_{m}$, $b_{n}$, $c_{p}$, $d_{q}$, respectively. Then the principal symbol~$M$ of the operator matrix $T_0$ is the matrix consisting of the principal symbols of its entries, see e.g.\ \cite{Agranovich90}, and is thus given by 
	\begin{align}\label{M(x,xi)}
	M(x,\xi) := \begin{pmatrix} 
	         a_{m}(x,\xi) & b_{n}(x,\xi) \\[2ex]
	         c_{p}(x,\xi) & d_{q}(x,\xi) 
	         \end{pmatrix}, 
	         \qquad (x,\xi)\in\RR^N\!\times\RR^N.
	\end{align}

\begin{ass}\label{ass:A}
	$T_0$ is \emph{uniformly Douglis-Nirenberg elliptic on $\RR^N$}, i.e. 
	\[
	|\det M(x,\xi)| \gtrsim \langle \xi \rangle^{\kappa}, \qquad x\in\RR^N, \ |\xi|\gtrsim 1, 
	\]
	where $\kappa:=\max\{m+q,n+p\} \ (>0)$, see e.g.\ \cite{Agranovich90}. 
\end{ass}

\begin{remark}
Since the orders of both diagonal entries of $T_0$ are positive by \eqref{orders}, Assumption~\ref{ass:A} implies that $T_0-\la$ is uniformly Douglis-Nirenberg elliptic on $\RR^N$ for all $\la\in\CC$.
\end{remark}

If $m+q\!\ne n+p$, an equivalent characterization of the Douglis-Nirenberg ellip\-ticity of $T_0$ on $\RR^N$ in terms of the ellipticity of either its diagonal or off-diagonal matrix entries can be given. 

\begin{lemma}\label{lem:equiv.charac}
	{\rm{(i)}} If $m+q>n+p$, then $T_0$ is uniformly Douglis-Nirenberg elliptic on $\RR^N$ if and only if both $T_a$ and $T_d$ are uniformly elliptic on $\RR^N$. 
	
	\smallskip
	\noindent
	{\rm{(ii)}} If $m+q<n+p$, then $T_0$ is uniformly Douglis-Nirenberg elliptic on $\RR^N$ if and only if both $T_b$ and $T_c$ are uniformly elliptic on $\RR^N$.
\end{lemma}
\begin{proof}
We prove claim (i); the proof of claim (ii) is completely analogous. It is easy to verify that
	\begin{equation}\label{key.ineq.japanese}
	\langle\xi\rangle^{m+q} \geq 2 \langle\xi\rangle^{n+p}, \quad |\xi|\gtrsim 1.
	\end{equation}
First let $T_0$ be uniformly Douglis-Nirenberg elliptic on $\RR^N\!$. Then, using the hypo\-theses $d_{q}\in S^{q}(\RR^N\!\times\RR^N)$, $b_{n}\in S^{n}(\RR^N\!\times\RR^N)$, $c_{p}\in S^{p}(\RR^N\!\times\RR^N)$ and \eqref{key.ineq.japanese}, we~get
	\begin{align*}
	\langle\xi\rangle^{m+q} 
	&\lesssim  |a_m(x,\xi)||d_q(x,\xi)|+|b_n(x,\xi)||c_p(x,\xi)| \\
	&\lesssim |a_m(x,\xi)|\langle\xi\rangle^q+\langle\xi\rangle^{n+p}
	\leq |a_m(x,\xi)|\langle\xi\rangle^q+\frac{1}{2}\langle\xi\rangle^{m+q}, \quad x\in\RR^N, \ |\xi|\gtrsim 1.
	\end{align*}  
\vspace{-2mm}Therefore,
	\begin{align}
	|a_m(x,\xi)|\gtrsim\frac{1}{2}\langle\xi\rangle^m, \quad x\in\RR^N, \ |\xi|\gtrsim 1.
	\end{align}
Since $a_{m}\in S^{m}(\RR^N\!\times\RR^N)$, the uniform ellipticity of $T_d$ follows in the same way.  
		
Now let $T_a$ and $T_d$ be uniformly elliptic on $\RR^N\!$. Then, using $b_{n}\in S^{n}(\RR^N\!\times\RR^N)$, $c_{p}\in S^{p}(\RR^N\!\times\RR^N)$ and the ellipticity of $T_a$ together with \eqref{key.ineq.japanese}, we obtain
	\begin{align*}
	\Bigl|\frac{b_n(x,\xi)c_p(x,\xi)}{a_m(x,\xi)}\Bigr| \lesssim \frac{\langle\xi\rangle^{n+p}}{\langle\xi\rangle^m}\leq \frac{1}{2}\langle\xi\rangle^q,\quad x\in\RR^N, \ |\xi|\gtrsim 1.
	\end{align*}
Therefore,
	\begin{align*}
	|a_m(x,\xi)d_q(x,\xi)-b_n(x,\xi)c_p(x,\xi)|&\geq |a_m(x,\xi)|\Bigg(|d_q(x,\xi)|-\Bigl|\frac{b_n(x,\xi)c_p(x,\xi)}{a_m(x,\xi)}\Bigr|\Bigg)\\[-1.1mm]
	&\gtrsim\langle\xi\rangle^m\Bigl(\langle\xi\rangle^q-\frac{1}{2}\langle\xi\rangle^q\Bigr)\\
	&=\frac{1}{2}\langle\xi\rangle^{m+q}, \quad x\in\RR^N, \ |\xi|\gtrsim1,
	\end{align*}
\nopagebreak
i.e.\ $T_0$ is uniformly Douglis-Nirenberg elliptic on $\RR^N$.
\end{proof}

Schur complements have proven to be useful tools in studying spectral properties of abstract operator matrices, see e.g.\ \cite[Section 2.2]{Tre08}. For $\la\in\CC\setminus\sigma(\ov{T_d})$, the (first) Schur complement of the operator matrix $T_0$ in \eqref{T0} is a pseudo-differential operator $S_1(\la):L^2(\RR^N)\to L^2(\RR^N)$, given by 
	\begin{align}\label{S_0}
	S_1(\la) := T_a -\la-T_b(T_d-\la)^{-1}T_c, \quad \Dom(S_1(\la)):=\sS(\RR^N).
	\end{align}
Note that 
	\begin{align}\label{symbol.class.S_0}
	S_1(\la)\in \Psi_{cl}^{\kappa-q}(\RR^N) = 
	          \begin{cases}
             \Psi_{cl}^{m}(\RR^N) \quad &\text{if} \quad m+q\geq n+p, \\[1.5ex]
             \Psi_{cl}^{n+p-q}(\RR^N) \quad &\text{if} \quad m+q< n+p.
            \end{cases} 
	\end{align}
Moreover, $S_1(\la)$ is closable since $\sS(\RR^N)$ is contained in $\Dom(S_1(\la)^*)$ and is dense in $L^2(\RR^N)$; we denote the closure of $S_1(\la)$ by $S(\la)$.

An important consequence of Assumption~\ref{ass:A} is the following result. In the proof we use Lemma~\ref{lem:equiv.charac} without mentioning it.

\begin{lemma}\label{lem:Schur.elliptic}
	Let Assumption~\ref{ass:A} be satisfied. Then $S(\la)$ is uniformly elliptic on~$\RR^N$ for every $\la\in\CC\setminus\sigma(\ov{T_d})$. 
\end{lemma}

\begin{proof}
It is not difficult to see that the principal symbol $\sigma_{\kappa-q}$ of $S(\la)$ is independent of $\la$. In fact, we have      
	\begin{align*}
	\sigma_{\kappa-q}(x,\xi)= 
	           \begin{cases}
             a_m(x,\xi) \quad &\text{if} \quad m+q>n+p,\\[2ex]
             \ds a_{m}(x,\xi)-\frac{b_{n}(x,\xi)c_p(x,\xi)}{d_{q}(x,\xi)}  &\text{if} \quad m+q=n+p,\\[2.5ex]     
             \ds -\frac{b_{n}(x,\xi)c_p(x,\xi)}{d_{q}(x,\xi)} \quad &\text{if} \quad m+q<n+p,
            \end{cases}  
	\quad \raisebox{1ex}{$(x,\xi)\in\RR^N\!\!\times\!\RR^N$\!.}																																																
	\end{align*}                          
Therefore, if $m+q>n+p$, then the uniform ellipticity of $T_a$ implies that 
	\[
	|\sigma_{\kappa-q}(x,\xi)|=|a_m(x,\xi)| \gtrsim \langle \xi \rangle^{m}, \qquad x\in\RR^N, \ |\xi|\gtrsim 1.
	\]
Now consider the case $m+q=n+p$. Since $T_0$ is uniformly Douglis-Nirenberg elliptic on $\RR^N$, we have  
	\begin{equation}\label{D.N.ellip.help}
	|\det M(x,\xi)| \gtrsim \langle \xi \rangle^{m+q}, \qquad x\in\RR^N, \ |\xi|\gtrsim 1. 
	\end{equation}
In view of $d_{q}\in S^{q}(\RR^N\!\times\RR^N)$, it immediately follows from \eqref{D.N.ellip.help} that   
	\begin{align*}
	|\sigma_{\kappa-q}(x,\xi)| = \frac{|\det M(x,\xi)|}{|d_{q}(x,\xi)|} \gtrsim \langle \xi \rangle^{m}, \qquad x\in\RR^N, \ |\xi| \gtrsim 1.
	\end{align*}
Finally, if $m+q<n+p$, then the ellipticity of the pseudo-differential operators $T_b$ and $T_c$ together with the assumption $d_{q}\in S^{q}(\RR^N\!\times\RR^N)$ yield
	\begin{align*}
	|\sigma_{\kappa-q}(x,\xi)| = \Bigl|\frac{b_{n}(x,\xi)c_p(x,\xi)}{d_{q}(x,\xi)}\Bigr| \gtrsim \frac{\langle \xi \rangle^{n}\langle \xi \rangle^{p}}{\langle \xi \rangle^{q}}= 
	\langle \xi \rangle^{n+p-q}, \qquad x\in\RR^N, \ |\xi|\gtrsim 1.
	\end{align*}
Altogether, we thus obtain
	\begin{align*}
	|\sigma_{\kappa-q}(x,\xi)| \gtrsim 
	     \begin{cases}
		    \langle \xi \rangle^{m} \quad &\text{if} \quad m+q \geq n+p,\\[2ex]
		    \langle \xi \rangle^{n+p-q}\quad &\text{if} \quad m+q<n+p,
		   \end{cases}
	\quad x\in\RR^N, \ |\xi|\gtrsim 1.		
	\end{align*}
Hence, in any case, $S(\la)$ is uniformly elliptic on $\RR^N$, see \eqref{symbol.class.S_0}. 
\end{proof}

Since, by Assumption~\ref{ass:A}, the pseudo-differential operator $S(\la)$ is uniformly elliptic 
for every fixed $\la\in\CC\setminus\sigma(\ov{T_d})$, it immediately follows that 
	\begin{align}\label{dom:S(la)}
	\Dom(S(\la))= H^{\kappa-q}(\RR^N) =  
	\begin{cases}
	H^{m}(\RR^N) \quad &\text{if} \quad m+q\geq n+p,\\ 
	H^{n+p-q}(\RR^N) \quad &\text{if} \quad m+q<n+p,
	\end{cases}
	\end{align}
see e.g.\ \cite[Chapter~14]{Wong-14b}. Furthermore, $S(\la)$ has a \emph{parametrix}, i.e.\ there exists an everywhere defined pseudo-differential operator $S_{\rm{p}}(\la)$ in $L^2(\RR^N)$ such that  
	\begin{align}\label{S.parametrix}
	 S_{\rm{p}}(\la)\in \Psi^{-\kappa+q}(\RR^N)=
	           \begin{cases}
                    \Psi^{-m}(\RR^N) \quad &\text{if} \quad m+q\geq n+p,\\
                    \Psi^{-n-p+q}(\RR^N) \quad &\text{if} \quad m+q<n+p,                 
                   \end{cases}
	\end{align}
and everywhere defined pseudo-differential operators $L_\la, R_\la\in\Psi^{-\infty}(\RR^N\!\times\RR^N)$ in $L^2(\RR^N)$ such that the identities  
	\begin{align}
	\label{parametrix-for-S}
	S_{\rm{p}}(\la) S(\la) = I+L_\la, \quad S(\la)S_{\rm{p}}(\la) = I+R_\la
	\end{align}
hold on $\Dom(S(\la))$ and $L^2(\RR^N)$, respectively, see e.g.\ \cite[Chapter~10]{Wong-14b}. Note that $S_{\rm{p}}(\la):L^2(\RR^N)\to L^2(\RR^N)$ is bounded by Theorem~\cite[Theorem~12.9]{Wong-14b} since $-\kappa+q \le -m \le 0$.

\section{Some auxiliary results}
Recall that an operator $A\in\mathscr{C}(X)$ is said to have a \textit{left approximate inverse} if there are a bounded operator $R_{\ell}:X\to X$ and a 
compact operator $K_X:X\to X$ such that $I_X+K_X$ extends $R_{\ell}A$, see e.g.\ \cite{EE87}. One can define \textit{right approximate inverses} in a similar way. An operator that is both a left and a right approximate inverse is referred to as a \textit{two-sided approximate inverse}.

\begin{remark}\label{rem:Schwartz.enough}
To check that $R_{\ell}$ is a left approximate inverse of $A$, it suffices to verify the equality $I_X+K_X=R_{\ell}A$ on any core of $A$. 

To see this, let $\cD_A\subset X$ be a core of $A\in\sC(X)$ and let $A_0$ be the restriction of $A$ to $\cD_A$. Let $x\in \Dom(A)$ be arbitrary. Then there is a sequence $\{x_k\}_{k=1}^\infty\subset \cD_A$ such that $x_k\to x$ and $A_0x_k\to Ax$ as $k\to\infty$ in $X$. Since $R_\ell$ and $I_X+K_X$ are bounded, it follows that 
	\[
	\ds R_\ell Ax=\lim_{k\to\infty}R_\ell A_0x_k=\lim_{k\to\infty}(I_X+K_X)x_k=(I_X+K_X)x.
	\]
\end{remark}

The proof of the following fact can be easily read off from the proofs of \cite[Theorems I.3.12-13]{EE87} and \cite[Lemma I.3.12]{EE87}.

\begin{proposition}\label{prop:left.apprx.inv}
	If $A\in\mathscr{C}(X)$ has a left approximate inverse, then $A$ has closed range and finite nullity $\nl\, A$.
\end{proposition}

Let $\la\in\CC\setminus\sigma(\ov{T_d})$ be such that $S(\la)$ is Fredholm and let $P_\la$ and $I-Q_\la$ be the orthogonal projections onto $\Ker(S(\la))$ and $\Ran(S(\la))$, respectively. Define $\wt{S}(\la)$ to be the restriction of $S(\la)$ to the subspace $\Dom(S(\la)) \cap \Ker(S(\la))^{\bot}$.

It is easy to see that the operator $S^\dagger(\la):L^2(\RR^N)\to L^2(\RR^N)$ given by
	\begin{align}\label{def:gen.inv}
	S^\dagger(\la) f:=\wt{S}(\la)^{-1}(I-Q_\la)f, \quad f\in L^2(\RR^N),
	\end{align}
is well-defined, bounded, and obeys the relations
	\begin{align}
	\label{def:gen.inv-1}
	S^\dagger(\la) S(\la) = I-P_{\la}, \quad  S(\la)S^\dagger(\la) = I-Q_{\la}
	\end{align}
on $\Dom(S(\la))$ and $L^2(\RR^N)$, respectively. The operator $S^\dagger(\la) :L^2(\RR^N)\to L^2(\RR^N)$ is called the \textit{generalized inverse} of $S(\la)$, 
see e.g.\ \cite{GGK-90b1}. Clearly, $S^\dagger(\la)$ is a two-sided approximate inverse of $S(\la)$ since $P_\la$ and $Q_\la$ have finite-rank. 

Note that we can not view $S_{\rm{p}}(\la)$ as a left approximate inverse of $S(\la)$ since  pseudo-differential operators of negative orders on $\RR^N$ are in general not compact in Sobolev spaces over $\RR^N$, see e.g.\ \cite{Agranovich90}. 
 
The following simple relationship between a parametrix $S_{\rm{p}}(\la)$ of $S(\la)$ and the generalized inverse $S^\dagger(\la)$ of $S(\la)$ will play a crucial role in this paper.

\begin{lemma}\label{lem:param-vs-apprx.inv}
	Let Assumption~\ref{ass:A} be satisfied and $\la\in\CC\setminus\sigma(\ov{T_d})$. Then, on $L^2(\RR^N)$,
	\begin{align}
	\label{rep:param-vs-apprx.inv}
	S^\dagger(\la)= S_{\rm{p}}(\la)(I-Q_\la) + L_\la S^\dagger(\la) R_\la - L_\la(I-P_\la) S_{\rm{p}}(\la).
	\end{align}
\end{lemma}
\begin{proof}
Replacing the operators $I-P_\la$ and $I-Q_\la$ on the right-hand side of \eqref{rep:param-vs-apprx.inv} by the respective operators on the left-hand sides in \eqref{def:gen.inv-1}, and taking \eqref{parametrix-for-S} into account, one immediately verifies the claim.
\end{proof}

\begin{lemma}\label{lem:kernel-smooth}
	Let Assumption~\ref{ass:A} be satisfied and let $\la\in\CC\setminus\sigma(\ov{T_d})$. Then
	\begin{enumerate}[font=\upshape,label=\upshape(\roman*\upshape)]
	\item $(\phi_1, \phi_2)^t \!\in\! \Ker(T-\la)$ implies $\phi_1 \in \Ker(S(\la))$ and $\phi_2=-(T_d-\la)^{-1}T_c\phi_1$;
	\item $(\psi_1, \psi_2)^t \!\in\! \Ker(T^*-\ov\la)$ implies $\psi_1 \in \Ker(S(\la)^*)$ and $\psi_2=-  (T_d^*-\ov{\la})^{-1}T_b^*\psi_1$.
	\end{enumerate}
Furthermore, in either case, we have  
	\begin{align}\label{H-alpha-1}
	(\phi_1, \phi_2)^t \!\in\! \sS(\RR^N) \oplus \sS(\RR^N).
	\end{align}      
\end{lemma}

\begin{proof}
We prove claim (i) and {\rm{\eqref{H-alpha-1}}}; claim (ii) can be shown in the same way. Let $\Phi \!:=\! (\phi_1, \phi_2)^t\in \Ker(T-\la)$ 
be arbitrary. Then, for all $\Psi\!:=\!(\psi_1,\psi_2)\in \sS(\RR^N)\oplus \sS(\RR^N)\subset\Dom(T_0^*)$, we have $0=\langle (T-\la)\Phi, \Psi\rangle=\langle\Phi,(T_0^*-\ov{\la})\Psi\rangle$ or, equivalently,
	\begin{align}\label{id:ker-1}
	\langle\phi_1, (T_a^*-\ov\la)\psi_1+T_c^*\psi_2\rangle + \langle\phi_2, T_b^*\psi_1+(T_d^*-\ov\la)\psi_2\rangle = 0;
	\end{align}
here we have used that  
	\begin{align*}
	T_0^* := \begin{pmatrix} T_a^* & T_c^* \\ T_b^* & T_d^* \end{pmatrix}  \quad \text{on} \quad \Dom(T_0)=\sS(\RR^N)\oplus \sS(\RR^N).
	\end{align*}
Choosing $\psi_2 = -(T_d^*-\ov{\la})^{-1}T_b^*\psi_1$ for $\psi_1\in \sS(\RR^N)$ in \eqref{id:ker-1}, we get $\langle\phi_1, S_1(\la)^*\psi_1\rangle\!=\!0$ for all $\psi_1\in \sS(\RR^N)$. Therefore, $\phi_1\in \Dom(S(\la))$ and $S(\la)\phi_1=0$. Consequently, using the first relation in \eqref{parametrix-for-S}, we find
	\[
	(I+L_\la\,)\phi_1=S_{\rm{p}}(\la) S(\la)\phi_1=0
	\]
or $\phi_1=-L_\la\,\phi_1$. Since $L_\la\,$ is a pseudo-differential operator with symbol from $S^{-\infty}(\RR^N\!\times\RR^N)$, it follows that $\phi_1\in H^{\alpha}(\RR^N)$ for every $\alpha\in\RR$, see e.g.\ \cite{Wong-14b}. Hence 
	\begin{equation}\label{phi1-is-schwartz}
	\phi_1\in\bigcap_{k\in\RR}H^{k}(\RR^N)=\sS(\RR^N). \vspace{-1.5mm}
	\end{equation}

On the other hand, choosing $\psi_1=0$ in \eqref{id:ker-1}, we obtain 
	\begin{align}\label{id:ker-2}
	\langle\phi_1, T_c^*\psi_2\rangle + \langle\phi_2, (T_d^*-\ov\la)\psi_2\rangle = 0
	\end{align}
for all $\psi_2\in \sS(\RR^N)$. In particular, for $\psi_2=(T_d^*-\ov{\la})^{-1}\varphi$ with arbitrary $\varphi\in \sS(\RR^N)$, 
	\begin{align}\label{id:ker-3}
	\langle\phi_1, T_c^*(T_d^*-\ov{\la})^{-1}\varphi\rangle + \langle\phi_2, \varphi\rangle = 0.
	\end{align}
Because $\phi_1 \in \sS(\RR^N)$ by \eqref{phi1-is-schwartz} and $(T_c^*(T_d^*-\ov{\la})^{-1})^*=(T_d-\la)^{-1}T_c$ on $\sS(\RR^N)$, we conclude from \eqref{id:ker-3} that $\langle (T_d-\la)^{-1}T_c\phi_1+\phi_2, \varphi\rangle=0$.
Since $\varphi\in \sS(\RR^N)$ was arbitrary, the density of $\sS(\RR^N)$ in $L^2(\RR^N)$ thus yields $\phi_2=-(T_d-\la)^{-1}T_c\phi_1$, completing the proof of claim (i). 

Now \eqref{H-alpha-1} immediately follows from \eqref{phi1-is-schwartz} since $(T_d-\la)^{-1}T_c$ is a pseudo-differential operator (of order $p-q$) and hence also $\phi_2=-(T_d-\la)^{-1}T_c\phi_1\in\sS(\RR^N)$.
\end{proof}

\section{Analytic description of the essential spectrum}

In this section we derive the characterization of the essential spectrum of the clo\-sure $T$ of the matrix pseudo-differential 
operator $T_0$ in \eqref{T0}. First we relate the Fredholm properties of $T-\la$ to those of the closure $S(\la)$ 
of its Schur complement.

\begin{theorem}\label{thm:approx.inv}
Let Assumption~\ref{ass:A} be satisfied and let $\la\in\CC\setminus\sigma(\ov{T_d})$. 
Further, assume $T-\la$ is Fredholm and let $T^\dagger(\la)$ be the generalized inverse of $\,T-\la$. Then the bounded operator $S_{\ell}(\la):L^2(\RR^N)\to L^2(\RR^N)$ defined by
	\begin{equation}\label{left.approx.inv.S}
	 S_{\ell}(\la)f:=P_1T^\dagger(\la)\binom{f}{0},\quad f\in \Dom(S_{\ell}(\la)):=L^2(\RR^N),
	\end{equation}
is a left approximate inverse of $S(\la)$, where $P_1:L^2(\RR^N)\oplus L^2(\RR^N)\to L^2(\RR^N)$ is the projection onto the first component. 
\end{theorem}

\pagebreak
\begin{proof}
Take an arbitrary $u\in\sS(\RR^N)$ and define $v:=-(T_d-\la)^{-1}T_cu \in \sS(\RR^N)$. Then, by definition of $v$,  
	\begin{equation}\label{lemschur:<=1}
	(T -\la) 
	\begin{pmatrix}
	u \\ v
	\end{pmatrix}
	=
	\begin{pmatrix}
	S(\la)u \\ 0
	\end{pmatrix}.
	\end{equation}
Since $T^\dagger(\la)$ is the generalized inverse of $T-\la$, it follows that $T^\dagger(\la)(T-\la)=I-\wh P_\la$ where $\wh P_\la$ is the orthogonal projection onto $\Ker(T-\la)$, see \eqref{def:gen.inv-1}. Applying $T^\dagger(\la)$ to \eqref{lemschur:<=1} we find
	\begin{equation}\label{left.inv.T}
	\begin{pmatrix}
	u \\ v
	\end{pmatrix}
	-\wh P_\la
	\begin{pmatrix}
	u \\ v
	\end{pmatrix}
	=T^\dagger(\la)(T-\la)
	\begin{pmatrix}
	u \\ v
	\end{pmatrix}
	=T^\dagger(\la)
	\begin{pmatrix}
	S(\la)u \\ 0
	\end{pmatrix}.
	\end{equation}
Since $T-\la$ is Fredholm, we have $k:=\dim\Ker(T-\la)<\infty$.
Let $\{(f_j,g_j)^t\}_{j=1}^k$ be an orthonormal basis for $\Ker(T-\la)$. By \eqref{H-alpha-1}, it is clear that 
	\begin{align}\label{phi_j.in.L2}
	\varphi_j:=T_c^*(T_d^*-\ov{\la})^{-1}g_j-f_j\in\sS(\RR^N)\subset L^2(\RR^N), \quad j=1,2,\ldots,k.
	\end{align}
If we define the operator $\wh{K}(\la):L^2(\RR^N) \to L^2(\RR^N)$ as
	\begin{align*}
	\wh{K}(\la)f := \sum_{j=1}^k \langle f, \varphi_j\rangle f_j, \quad f\in\Dom(\wh{K}(\la))=L^2(\RR^N),
	\end{align*}
then \eqref{phi_j.in.L2} implies that $\wh{K}(\la):L^2(\RR^N) \to L^2(\RR^N)$ is compact. Moreover, it can be easily checked that 
	\begin{align}\label{whK.compact.2}
	\wh{K}(\la)f=P_1\wh P_\la
	\begin{pmatrix}
	-f \\[1ex]
	(T_d-\la)^{-1}T_cf
	\end{pmatrix}, \quad f\in\sS(\RR^N).
	\end{align}
Hence applying $P_1$ to \eqref{left.inv.T}, using \eqref{whK.compact.2} and \eqref{left.approx.inv.S}, we arrive at 
	\[
	u+\wh{K}(\la)u=S_{\ell}(\la)S(\la)u.
	\]
Since $u\in\sS(\RR^N)$ was arbitrary and $\sS(\RR^N)$ is a core of $S(\la)$, in view of Remark~\ref{rem:Schwartz.enough} it follows that $S_{\ell}(\la)$ is a left approximate inverse of $S(\la)$.
\end{proof}

Our basic assumptions imply that $(T_d-\la)^{-1}T_c$ and $T_b(T_d-\la)^{-1}$ are pseudo-differential operators in $L^2(\RR^N)$ defined on the Schwartz space $\sS(\RR^N)$ with symbols in $S^{p-q}(\RR^N\!\times\RR^N)$ and $S^{n-q}(\RR^N\!\times\RR^N)$, respectively. In the sequel, we work with the following extensions of these operators in $L^2(\RR^N)$:
	\begin{align}
	\label{def:F1}
	F_1(\la) &:= (T_d-\la)^{-1}T_c, \qquad \Dom(F_1(\la)) := H^{p-q}(\RR^N) \cap L^2(\RR^N), \\ \label{def:F2}
	F_2(\la) &:= T_b(T_d-\la)^{-1}, \qquad \Dom(F_2(\la)) :=  H^{n-q}(\RR^N) \cap L^2(\RR^N). 
	\end{align}

\begin{ass}\label{ass:B}
If the off-diagonal orders $n$ and $p$ have opposite sign, i.e.\ if $\min\{n,p\}<0<\max\{n,p\}$, let $m+q\ge \max\{n,p\}$. 
\end{ass}

\begin{remark}
(i) Assumption \ref{ass:B} is equivalent to $\kappa\!\ge\! \max\{n,p\}$ for $\kappa \!=\! \max\{m\!+\!q, $ $n\!+\!p\}$ in \vspace{0.5mm}Assumption \ref{ass:A}. 

\noindent
(ii) If $n$ and $p$ have the same sign, i.e.\ if $\min\{n,p\} \ge 0$ or $\max\{n,p\}\le 0$, no condition needs to be imposed.

\end{remark}
	
\begin{lemma}\label{lem:dom.inclusion.SFF}
	Let Assumptions~\ref{ass:A}, \ref{ass:B} be satisfied. Then, for every $\la\in\CC\setminus\sigma(\ov{T_d})$,
	\begin{equation}\label{dom.inclusion.SFF}
	\Dom(S(\la)) \subset \Dom(F_1(\la)) \cap \Dom(F_2(\la)).
	\end{equation}
\end{lemma}
\begin{proof}
Let $\la\in\CC\setminus\sigma(\ov{T_d})$ be fixed. Note that 
	\begin{align}
	\ds\Dom(F_1(\la)) \cap \Dom(F_2(\la))=
									\begin{cases}
									H^{p-q}(\RR^N)\cap L^2(\RR^N) \quad &\text{if} \quad p\geq n,\\
									H^{n-q}(\RR^N)\cap L^2(\RR^N) \quad &\text{if} \quad p< n.
									\end{cases}
	\end{align}
By Assumption \ref{ass:B} 
on $m$, $n$, $p$, $q$, it follows that 
$\kappa \ge \max\{n,p\}$ and thus $\kappa-q \ge \max\{n-q,p-q\}$.
Hence, in view of \eqref{dom:S(la)}, we obtain \eqref{dom.inclusion.SFF}.
\end{proof}

\begin{theorem}\label{thm2:approx.inv}
	Let Assumptions~\ref{ass:A}, \ref{ass:B} be satisfied and let $\la\in\CC\setminus\sigma(\ov{T_d})$. Further, assume $S(\la)$ is Fredholm and let $S^\dagger(\la)$ be the generalized inverse of $S(\la)$. Then the operator $T_{\ell}(\la):L^2(\RR^N)\oplus L^2(\RR^N)\to L^2(\RR^N)\oplus L^2(\RR^N)$ defined to be the bounded extension to $L^2(\RR^N)\oplus L^2(\RR^N)$ of the operator matrix  
	\begin{align}\label{left.approx.inv.op.mat}
	&T^{\#}(\la) := \begin{pmatrix}
	S^\dagger(\la)\, & \quad -S^\dagger(\la)\,F_2(\la)\\[1ex]
	-F_1(\la)S^\dagger(\la)\, & \quad (T_d-\la)^{-1} + F_1(\la)S^\dagger(\la)\,F_2(\la)
	\end{pmatrix},\\[1ex]
	&\Dom(T^{\#}(\la)):= L^2(\RR^N)\oplus \Dom(F_2(\la)), 
	\end{align}
is a left approximate inverse of $\,T-\la$, where the operators $F_1(\la)$ and $F_2(\la)$ are defined as in \eqref{def:F1}-\eqref{def:F2}. 
\end{theorem}
\begin{proof}
First of all, we note that the operator $T^{\#}(\la)$ is well-defined since $\Ran(S^\dagger(\la))$ is a subset of $\Dom(S(\la))$ and $\Dom(S(\la))\subset\Dom(F_1(\la))$ by \eqref{dom.inclusion.SFF}. Further, for $(f,g)^{t} \in \Ran(T_{0}-\la)\subset\sS(\RR^N)\oplus \sS(\RR^N)$, there exists $(u,v)^{t}\in\Dom(T_{0})=\sS(\RR^N)\oplus\sS(\RR^N)$ such that
	\begin{align}\label{eshmat1}
	\begin{pmatrix}
	f\\g
	\end{pmatrix}
	=(T_{0}-\la) \begin{pmatrix}
	u\\v
	\end{pmatrix}
	\quad \Longleftrightarrow \quad 
	\begin{matrix}
	(T_a-\la)u+T_bv=f,\\
	T_cu+(T_d-\la)v=g,
	\end{matrix}
	\end{align}
which implies that  
	\begin{align}\label{Schur-u}	
	S_1(\la)u = f-T_b(T_d-\la)^{-1}g.
	\end{align}
Multiplying \eqref{Schur-u} from the left by $S^\dagger(\la)$ and using the first relation in \eqref{def:gen.inv-1}, we~find
	\begin{align*}
	u=S^\dagger(\la)f-S^\dagger(\la)F_2(\la)g+P_{\la}u,
	\end{align*}
where $P_\la$ is the orthogonal projection in $L^2(\RR^N)$ onto $\Ker(S(\la))$. Since $S(\la)$ is Fredholm by assumption, $P_\la$ has finite rank and is thus compact. Inserting this representation of $u$ into the second equation in \eqref{eshmat1} and solving for $v$, we find
	\begin{align*}
	v= (T_d-\la)^{-1}g-F_1(\la)S^\dagger(\la)f+F_1(\la)S^\dagger(\la)F_2(\la)g-F_1(\la)
	P_\la u.
	\end{align*}
Therefore, by definition of $T^{\#}(\la)$ in \eqref{left.approx.inv.op.mat},
	\begin{align}\label{apprx.inv-1}
	\begin{pmatrix}
	u\\v
	\end{pmatrix}
	=T^{\#}(\la) \begin{pmatrix}
	f\\g
	\end{pmatrix}
	-K(\la)  \begin{pmatrix}
	u\\v
	\end{pmatrix}
	\end{align}
with the operator matrix $K(\la)$ in $L^2(\RR^N) \oplus L^2(\RR^N)$ defined by 
	\begin{align}
	&K(\la) := \begin{pmatrix}
	-P_{\la} & \: 0\\[1ex]
	F_1(\la)P_{\la} & \: 0
	\end{pmatrix}, \quad \Dom(K(\la)):=L^2(\RR^N)\oplus L^2(\RR^N).
	\end{align}
Since $P_\la:L^2(\RR^N)\to L^2(\RR^N)$ is everywhere defined and has finite-rank with $\Ran(P_\la)=\Ker(S(\la)) \subset \Dom(F_1(\la))$, it follows that $F_1(\la)P_\la$ is a compact operator in $L^2(\RR^N)$. By \eqref{apprx.inv-1} and \eqref{eshmat1}, we have
	\begin{equation*}
	T^{\#}(\la)(T_0-\la) = I+K(\la) \quad \text{on} \quad \Dom(T_0)=\sS(\RR^N)\oplus\sS(\RR^N).
	\end{equation*} 
Since $\sS(\RR^N)\oplus \sS(\RR^N)$ is a core of $T-\la$, it is thus left to be shown that $T^{\#}(\la)$ has a bounded extension to $L^2(\RR^N)\oplus L^2(\RR^N)$, see Remark~ \ref{rem:Schwartz.enough}.

Using relation \eqref{rep:param-vs-apprx.inv} in  Lemma~\ref{lem:param-vs-apprx.inv}, we decompose $T^{\#}(\la)=T_1^{\#}(\la)+T_2^{\#}(\la)$ where
	\begin{align}
	T_1^{\#}(\la) :=& \begin{pmatrix}
	S_{\rm{p}}(\la) & \quad -S_{\rm{p}}(\la)F_2(\la)\\[1ex]
	-F_1(\la)S_{\rm{p}}(\la) & \quad (T_d-\la)^{-1} + F_1(\la)S_{\rm{p}}(\la)F_2(\la)
	\end{pmatrix}
	\intertext{and}
	T_2^{\#}(\la) :=& \begin{pmatrix}
	H(\la) & \quad -H(\la)F_2(\la)\\[1ex]
	-F_1(\la)H(\la) & \quad  F_1(\la)H(\la)F_2(\la)
	\end{pmatrix}
	\end{align}
with 
	\[
	\Dom(T_1^{\#}(\la))=\Dom(T_2^{\#}(\la))=L^2(\RR^N)\oplus \Dom(F_2(\la)) 
	\]
and
	\begin{align*}
	H(\la) := L_{\la}S^{\dagger}(\la)R_{\la}-L_{\la}(I-P_{\la})S_{\rm{p}}(\la)-S_{\rm{p}}(\la)Q_{\la}.
	\end{align*}
First we show that $T_1^{\#}(\la)$ has a bounded extension to $L^2(\RR^N)\oplus L^2(\RR^N)$. 
To this end, we recall that $S_{\rm{p}}(\la)$ is bounded in $L^2(\RR^N)$ by \eqref{S.parametrix} since $-\kappa+q \le -m< 0$, see \cite[Theorem 12.9]{Wong-14b}.
Further, we note that, by \cite[Theorem 8.1]{Wong-14b}, we have
	\begin{align}
	\label{symb.class.F2Sp}
	S_{\rm{p}}(\la) F_2(\la) \in \Psi^{-\kappa+n}(\RR^N) & \!=\!
	\begin{cases}
	\Psi^{-m+n-q}(\RR^N) \ &\text{if} \  m+q\geq n+p,\\
	\Psi^{-p}(\RR^N) \quad &\text{if} \ m+q<n+p,
	\end{cases}
	\\ \label{symb.class.F1Sp}
	F_1(\la)S_{\rm{p}}(\la)  \in \Psi^{-\kappa+p}(\RR^N) & \!=\!
	\begin{cases}
	\Psi^{p-q-m}(\RR^N) \quad &\text{if} \ m+q\geq n+p,\\
	\Psi^{-n}(\RR^N) \quad &\text{if} \ m+q<n+p,
	\end{cases}\\
\intertext{and}
  \label{symb.class.F1F2Sp}
	\hspace{-3mm}
	F_1(\la) S_{\rm{p}}(\la) F_2(\la) \in \Psi^{-\kappa+n+p-q}(\RR^N) & \!=\!
	\begin{cases}
	\Psi^{p-m+n-2q}(\RR^N) \!\!\!&\text{if} \ m+q\geq n+p,\\
	\Psi^{-q}(\RR^N) \quad &\text{if} \ m+q<n+p.
	\end{cases}
	\hspace{-3mm}
	\end{align}
By Assumption \ref{ass:B} on $m$, $n$, $p$, $q$, we have $\kappa \ge \max\{n,p\}$ which ensures that 
$S_{\rm{p}}(\la) F_2(\la)$ has a bounded extension to $L^2(\RR^N)$ and $F_1(\la)S_{\rm{p}}(\la)$ is bounded in $L^2(\RR^N)$. By definition, $\kappa=\max\{m+q,n+p\}$ so that $\kappa-(n+p)+q \ge q >0$ and hence 
$F_1(\la) S_{\rm{p}}(\la) F_2(\la)$ has a bounded extension to $L^2(\RR^N)$. Altogether, all entries of $T_1^{\#}(\la)$ have bounded extensions to or are bounded in $L^2(\RR^N)$.
	
The existence of a bounded extension of $T_2^{\#}(\la)$ to $L^2(\RR^N)\oplus L^2(\RR^N)$ can be shown similarly. Indeed, it is easy to see that $H(\la)$ has a bounded extension to 
$L^2(\RR^N)$. Furthermore, since $F_1(\la)L_\la\in\Psi^{-\infty}(\RR^N)$ and \eqref{symb.class.F1Sp} holds, it follows that 
$F_1(\la)H(\la)$ is a bounded operator on $L^2(\RR^N)$. In the same way, it follows that $H(\la)F_2(\la)$ has a bounded extension to $L^2(\RR^N)$. 
Finally, by the above observations and noting that $R_\la\,F_2(\la)\in\Psi^{-\infty}(\RR^N)$, we conclude that $F_1(\la)H(\la)F_2(\la)$ has a bounded 
extension to $L^2(\RR^N)$. Thus, again by \cite[Theorem 12.9]{Wong-14b}, it follows that $T_2^{\#}(\la)$ has a bounded extension to $L^2(\RR^N)\oplus L^2(\RR^N)$, and hence so does $T^{\#}(\la)$.
\end{proof}

The following result is one of the key ingredients for the proof of our main result.

\begin{proposition}\label{thm:Schur-complement}
	Let Assumptions~\ref{ass:A}, \ref{ass:B} be satisfied. Then, for $\la\in\CC\setminus\sigma(\ov{T_d})$,
	\[
	T-\la \; \text{is Fredholm} \quad \Longleftrightarrow \quad S(\la) \; \text{is Fredholm,}
	\] 
	and in this case $T-\la$ and $S(\la)$ have the same defect, $\df(T-\la)=\df(S(\la))$.
\end{proposition}

\begin{proof}
First let $T-\la$ be Fredholm. Then $T-\la$ has a two-sided approximate inverse, see \cite[Theorem~I.3.15 and Remark~I.3.27]{EE87} and thus, in particular, a left approximate inverse. Hence Proposition~\ref{thm:approx.inv} implies that $S(\la)$ also has a left approximate inverse and so, by Proposition~\ref{prop:left.apprx.inv}, $S(\la)$ has closed range and finite nullity.

On the other hand, we have $k:=\df(T-\la)<\infty$ since $T-\la$ is Fredholm. Let $\{\phi_j\}^k_{j=1} = \{(\phi_{j,1}, \phi_{j,2})^t\}_{j=1}^k$ be an orthonormal basis of $\Ran(T-\la)^\bot=\Ker(T^*-\ov\la)$. By Lemma~\ref{lem:kernel-smooth} (ii), $\phi_{j,1}\in \Ker(S(\la)^*)$ and $\phi_{j,2}=-(T_d^*-\ov{\la})^{-1}T_b^*\phi_{j,1}$ for all $j\in\{1,\ldots,k\}$. 
If $\sum_{j=1}^kc_j\phi_{j,1}=0$ for some constants $c_1,\ldots,c_k$, then 
	\begin{align*}
	\sum_{1\leq j \leq k}c_j\phi_{j,2}=-\sum_{1\leq j \leq k} c_j(T_d^*-\ov{\la})^{-1}T_b^*\phi_{j,1}=-(T_d^*-\ov{\la})^{-1}T_b^*\Bigl(\sum_{1\leq j \leq k} c_j\phi_{j,1}\Bigr)=0
	\end{align*}
and hence $\sum_{j=1}^kc_j\phi_j=0$. Because $\{\phi_j\}^k_{j=1}$ is a basis, it follows that $c_j=0$ for all $j\in\{1,\ldots,k\}$, i.e.\ $\{\phi_{j,1}\}_{j=1}^k$ are linearly independent. Hence $\dim\Ker(S(\la)^*)\geq k$.
If $\dim\Ker(S(\la)^*) \geq k+1$, then there would exist $\phi_{0,1}\in\Ker(S(\la)^*)$ such that $\{\phi_{j,1}\}^k_{j=0}$ are linearly independent. Since $\phi_{0,1}\in\Ker(S(\la)^*)$, it would then follow that $\phi_{0,1}\in \sS(\RR^N)$, $\phi_0:=(\phi_{0,1}, -(T_d^*-\ov{\la})^{-1}T_b^*\phi_{0,1})^t \in \Ker\,(T^*-\ov{\la})$, and $\{\phi_j\}^k_{j=0}$ would be linearly independent, contradicting the assumption that $\df(T-\la)=k$. Therefore, $\df(S(\la))=\dim\Ker(S(\la)^*) = k$ and $S(\la)$ is Fredholm.

\smallskip

Now let $S(\la)$ be Fredholm. Then $T-\la$ has a left approximate inverse, see Theorem~\ref{thm2:approx.inv}. By Proposition~\ref{prop:left.apprx.inv}, $T-\la$ has closed range and finite nullity. On the other hand, $k:=\df(S(\la))<\infty$ since $S(\la)$ is Fredholm. Let $\{\phi_{j,1}\}_{j=1}^k$ be an orthonormal basis of $\Ran(S(\la))^\bot = \Ker(S(\la)^*)$. Then the vectors $\{\phi_j\}^k_{j=1}\in \Ker\:(T^*-\ov{\la})$ given by $\phi_j:=(\phi_{j,1}, \phi_{j,2})^t$ with $\phi_{j,2}:=-(T_d^*-\ov{\la})^{-1}T_b^*\phi_{j,1}$ must be linearly independent, for otherwise $\{\phi_{j,1}\}_{j=1}^k$ would be linearly dependent, 
contradicting the assumption that $\df(S(\la))=k < \infty$. Hence $\dim\Ker(T^*-\ov{\la}) \geq k$ holds. If $\dim\Ker(T^*-\ov{\la}) \geq k+1$, there would exist a vector $\phi_0:=(\phi_{0,1},\phi_{0,2})^t\in \Ker(T^*-\ov{\la})$ such that $\{\phi_j\}^k_{j=0}$ are linearly independent. By Lemma~\ref{lem:kernel-smooth}, $\phi_{0,1} \in \Ker(S(\la)^*)$ and $\phi_{0,2} = -(T_d^*-\ov{\la})^{-1}T_b^*\phi_{0,1}$. This would imply that the vectors $\{\phi_{j,1}\}_{j=0}^{k}$ are linearly independent and hence $\dim \Ker(S(\la)^*) \geq k+1$. This contradiction yields that $\df(T-\la)=\dim\Ker(T^*-\ov{\la})=k$ and $T-\la$ is Fredholm.	
\end{proof}


Recall that the Grushin symbol class $S_0^k(\RR^N\!\times \RR^N)$, $k\in\RR$, 
consists of those symbols $\sigma \in S^k(\RR^N\!\times \RR^N)$ such that, for all multi-indices $\alpha, \beta\in \NN_0^N$, there is a positive function $x\mapsto C_{\alpha,\beta}(x)$, $x\in\RR^N$, satisfying 
	\[
	|(\partial^\beta_x\partial^\alpha_\xi\sigma)(x,\xi)| \leq C_{\alpha,\beta}(x)\langle \xi \rangle^{k-|\alpha|}, \quad 
	(x,\xi) \in \RR^N\!\times\RR^N,
	\]
and $\lim\limits_{|x|\to\infty}C_{\alpha,\beta}(x)=0$ if $\beta\neq0$, see \cite{Grushin-1970}.

\begin{ass}\label{ass:C}
	Suppose that the symbol $\sigma_\la^{1}$ of the first Schur complement $S_{1}(\la)$ satisfies $\sigma_\la^{1}\in S_0^{\kappa-q}(\RR^N\!\times\RR^N)$ for all $\la\in\CC\setminus\sigma(\ov{T_d})$, where $\kappa=\max\{m+q,n+p\}$, and there exists a limiting symbol $\sigma_{\la,\infty}^{1}\in S^{\kappa-q}(\RR^N\!\times\RR^N)$ such that $\sigma_{\la,\infty}^{1}$ is independent of $x$, i.e.\ $\sigma_{\la,\infty}^{1}(x,\xi)=\sigma_{\la,\infty}^{1}(\xi)$ for all $x\in\RR^N$, and 
	\begin{align}\label{wong-assumption}
	\lim\limits_{|x|\to\infty}\langle \xi \rangle^{-\kappa+q}|\sigma_\la^{1}(x,\xi)-\sigma_{\la,\infty}^{1}(\xi)|=0 
	\quad \text{ uniformly in } \xi\in\RR^N.
	\end{align}
\end{ass}	

The next theorem gives an analytic description of the essential spectrum of the closure $T\!=\!\overline{T_0}$ of the non-self-adjoint pseudo-differential matrix operator $T_0$ in~\eqref{T0}.

\begin{theorem}\label{thm:main}
	Let Assumptions~\ref{ass:A}, \ref{ass:B}, and \ref{ass:C} be satisfied. Then
	\begin{align}\label{sess(T)}
	\sess(T)\setminus\sigma(\ov{T_d})=\bigl\{\la\in\CC\setminus\sigma(\ov{T_d}):\, \sigma_{\la,\infty}^{1}(\xi)=0 \ \text{ for some } \xi\in\RR^N\bigr\},
	\end{align}
	where $\sigma_{\la,\infty}^{1}\in S^{\kappa-q}(\RR^N\!\times\RR^N)$ is the limiting symbol in the sense of \eqref{wong-assumption} of the symbol $\sigma_\la^{1}\in S_0^{\kappa-q}(\RR^N\!\times\RR^N)$ of the Schur complement given by \eqref{S_0}.	
\end{theorem}

\begin{proof}
For $\la\in\CC\setminus\sigma(\ov{T_d})$, Theorem~\ref{thm:Schur-complement} implies that
	\begin{equation}\label{Schur:1}
	\la\in\sess(T) \quad \Longleftrightarrow \quad 0\in\sess(S(\la)).
	\end{equation}
By Assumption~\ref{ass:A}, Lemma \ref{lem:Schur.elliptic} shows that $S(\la)$ 
is uniformly elliptic on $\RR^N$. Because of Assumption~\ref{ass:C}, we can apply \cite[Theorem~3.1]{Wong88-Communications} which yields 
	\begin{align}\label{Wong:2}
	0\in\sess(S(\la)) \quad \Longleftrightarrow \quad \sigma_{\la,\infty}^{1}(\xi)=0 \quad \text{for some} \quad \xi\in\RR^N.
	\end{align}
Combining \eqref{Schur:1} and \eqref{Wong:2}, we obtain the claim. 
\end{proof}

\begin{remark}
\label{last}
(i) If the second Schur complement $S_2(\la): L^2(\RR^N)\to L^2(\RR^N)$ defined by $S_2(\la):=T_d-\la-T_c(T_a-\la)^{-1}T_b$ on $\Dom(S_2(\la))\!:=\!\sS(\RR^N)$ for $\la \!\in\! \CC \setminus \sigma(\overline{T_a})$ 
satisfies the analogue of Assumption \ref{ass:C} with limiting symbol $\sigma_{\la,\infty}^2 \in S^{\kappa-m}(\RR^N\!\times\RR^N)$, one can obtain an analogue of Theorem~\ref{thm:main}:
  \begin{align}\label{sess(T)2}
	\sess(T)\setminus\sigma(\ov{T_a})=\bigl\{\la\in\CC\setminus\sigma(\ov{T_a}):\,\sigma_{\la,\infty}^2(\xi)=0 \ \text{ for some } \xi\in\RR^N\bigr\};
	\end{align}
note that the sets on the right-hand sides of \eqref{sess(T)}, \eqref{sess(T)2} must coincide outside of the set $\sigma(\ov{T_a}) \cup \sigma(\ov{T_d})$. Thus combining \eqref{sess(T)}, \eqref{sess(T)2} 
the exceptional set $\sigma(\ov{T_d})$ (or $\sigma(\ov{T_a})$, respectively) in the 
analytic description of the essential spectrum $\sess(T)$ can be reduced to the intersection $\sigma(\ov{T_a}) \cap \sigma(\ov{T_d})$.
\vspace{1mm}
	
	\noindent
	(ii) The reduction of the exceptional set from $\sigma(\ov{T_d})$ (or $\sigma(\ov{T_a})$, respectively) to $\sigma(\ov{T_a}) \cap \sigma(\ov{T_d})$ may be considerable, 
	as shown in the application in the next section; there $\sigma(\ov{T_d})$ is a line, while $\sigma(\ov{T_a}) \cap \sigma(\ov{T_d})$ consists of at most finitely many points and may even be empty.  
\end{remark}

\section{Application to the falling liquid film problem}
The spectral properties of the following matrix differential operator were used in examining the interaction of two-dimensional solitary pulses on falling liquid films in \cite{Pradas-Tseluiko-Kalliadasis-2011}. In the Hilbert space $L^2(\RR)\oplus L^2(\RR)$ we consider 
	\begin{equation}\label{T0Film}
	T_0\!:=\!\begin{pmatrix}
	\ds \frac{9\eta}{2\delta}D^2 \!+\! \phi_1 D \!+\! \phi_0 \!&\! \ds\psi_3 D^3 \!+\! \psi_2D^2 \!+\! \psi_1D \!+\! \psi_0\! \\[1.5ex]
	-D  & c_0 D 
	\end{pmatrix}\!\!, \; \Dom(T_0)\!:=\!\sS(\RR)\oplus\sS(\RR)
	\end{equation}
with $D = \d/\d x$ and $\phi_j,\psi_j\in B^\infty(\RR,\CC)$; here $B^\infty(\RR,\CC)$ denotes the space of all functions in $C^\infty(\RR,\CC)$ having bounded derivatives of arbitrary orders. It is assumed that the coefficient functions in \eqref{T0Film} have the following asymptotic behaviour:
	\begin{align*}
	& \phi_1(x)=c_0-\frac{17}{21}+\o(1), \;\  \phi_0(x)=-\frac{5}{2\delta}+\o(1), \\
	& \psi_3(x)=\frac{5}{6\delta}+\o(1), \;\  \psi_2(x)=-\frac{2\eta}{\delta}+\o(1), \;\ \psi_1(x)=\frac{1}{7}+\o(1), \;\ \psi_0(x)=\frac{5}{2\delta}+\o(1)
	\end{align*}
as $|x|\to\infty$, where $\delta, \eta, c_0 \in\RR$ are physical constants corresponding to the reduced Reynolds number, viscous-dispersion number and the speed at which solitary pulses propagate, respectively. 

\begin{theorem}
	The essential spectrum of the closure $T$ of $\,T_0$ in \eqref{T0Film} is given by
	\begin{equation}\label{ess.spec.film.1}
	\sess(T)\setminus\Lambda =\Bigl\{-\dfrac{1}{2}\alpha(\xi)\pm \dfrac{1}{2}\sqrt{\alpha(\xi)^2-4\beta(\xi)}: \: \xi\in\RR \Bigr\}\setminus\Lambda
	\end{equation}
	outside of a finite exceptional set of points $\Lambda\subset\ii\RR$ where, for $\xi\in\RR$, 
	\begin{equation}\label{alpha.beta}
	\begin{aligned}
	\alpha(\xi)\!&:=\!\frac{5}{2\delta}\!-\!\Bigl(2c_0\!-\!\frac{17}{21}\Bigr)(\ii\xi)\!-\!\frac{9\eta}{2\delta}(\ii\xi)^2, \\[0.5ex]
	\beta(\xi)\!&:=\!\frac{5}{6\delta}(\ii\xi)^4\!+\!\frac{9\eta}{2\delta}\Bigl(c_0-\frac{4}{9}\Bigr)(\ii\xi)^3\!+\!\Bigl(\frac{1}{7}\!+\!c_0\Bigl(c_0\!-\!\frac{17}{21}\Bigr)\Bigr)(\ii\xi)^2\!+\!
	\frac{5}{2\delta}(1\!-\!c_0)(\ii\xi).
	\end{aligned}
	\end{equation}
	\end{theorem}
\begin{proof}
Clearly, $T_0$ is uniformly Douglis-Nirenberg elliptic on $\RR$ since both off-diagonal entries are uniformly elliptic on $\RR$ in the usual sense and thus Assumption~\ref{ass:A} is satisfied, see Lemma~\ref{lem:equiv.charac}\,(ii). Moreover,  $\sigma(\ov{T_d})=\ii\RR$ and the first Schur complement is given by
	\begin{align*}
	S_1(\la) &=\sum_{j=0}^2 \phi_j D^j - \la - \Bigl(\sum_{k=0}^3\psi_k D^k\Bigr) (c_0 D-\la)^{-1} (-D), \quad \la\in\CC\setminus\ii\RR.
	\end{align*}
It is not difficult to see that the symbol of $S_1(\la)$ has the form
\begin{align*}
	\sigma_\la^{1}(x,\xi) = \sum_{j=0}^2 \phi_j (\ii\xi)^j - \la + \frac{\ii\xi}{c_0\ii\xi-\la} \sum_{k=0}^3\psi_k (\ii\xi)^k, \quad (x,\xi)\in\RR\times\RR,  
\end{align*}
and it belongs to the symbol class $S_0^3(\RR\times\RR)$ due to the asymptotic expansions of the functions $\phi_j$ and $\psi_k$. Further, it can be easily checked that Assumption~\ref{ass:C} is satisfied with the limiting symbol 
	\begin{equation}\label{film.lim.symbol}
	\sigma_{\la,\infty}^{1}(\xi) = \frac{1}{c_0\ii\xi-\la} \big(\la^2+\alpha(\xi)\la+\beta(\xi)\big), \quad \xi\in\RR,
	\end{equation}
where $\alpha(\xi)$ and $\beta(\xi)$ are given by \eqref{alpha.beta}. Therefore, Theorem~\ref{thm:main} applies and in view of \eqref{film.lim.symbol} it yields that, for the closure $T$ of $T_0$ in $L^2(\RR)\oplus L^2(\RR)$,
	\begin{equation}\label{ess.spec.film}
	\begin{aligned}
	\sess(T)\setminus\ii\RR &= \bigl\{\la\in\CC\setminus\ii\RR:\, \la^2+\alpha(\xi)\la+\beta(\xi)=0 \:\: \text{for some} \:\: \xi\in\RR\bigr\}\\
						&=\Bigl\{-\dfrac{1}{2}\alpha(\xi)\pm \dfrac{1}{2}\sqrt{\alpha(\xi)^2-4\beta(\xi)}: \: \xi\in\RR \Bigr\}\setminus\ii\RR.
	\end{aligned}
	\end{equation}
By the hypothesis on the coefficient functions we can also employ the second Schur complement approach in an analogous way to get
	\begin{equation}\label{ess.spec.film.11}
	\sess(T)\setminus\sigma(\ov{T_{a}}) =\Bigl\{-\dfrac{1}{2}\alpha(\xi)\pm \dfrac{1}{2}\sqrt{\alpha(\xi)^2-4\beta(\xi)}: \: \xi\in\RR \Bigr\}\setminus\sigma(\ov{T_{a}}),
	\end{equation}
so that we obtain the description \eqref{ess.spec.film.1} with exceptional set 
	\[
	\Lambda:=\sigma(\ov{T_{a}})\cap\sigma(\ov{T_{d}})=\sigma(\ov{T_{a}})\cap\ii\RR,
	\]
see Remark \ref{last} (i). Since $\ov{T_{a}}$ has asymptotically constant coefficients, $\sess(\ov{T_{a}})$ can be calculated explicitly. In fact, $\sess(\ov{T_{a}})$ is a horizontal parabola in the open left half-plane with vertex at $(-\frac{5}{2\delta},0)$. On the other hand, it is not \vspace{-1mm} difficult to show that the closure of the numerical range $\overline{W(\ov{T_{a}})}$ of $\ov{T_{a}}$ is contained in a horizontal sector in some left half-plane in $\CC$ with semi-angle $\theta\in(0,\pi/2)$ and vertex in $\RR$. By standard arguments, one can show that for sufficiently large $\mu_0\in[0,\infty)$ \vspace{-1mm} we have $\mu_0\in\rho(\ov{T_{a}})$ and hence 
$\sigma(\ov{T_{a}})\subset \ov{W(\ov{T_{a}})}$, see \cite[Theorem V.3.2]{Kato}. Altogether, it follows that $\Lambda\subset\ii\RR$ is a finite set.
\end{proof}

The locus of the curve describing the essential spectrum outside of the finite set $\Lambda$ in \eqref{ess.spec.film.1} is shown in Figures~\ref{fig:film} (a) and (b) for the physical parameters corresponding to water, i.e.\ $\delta=0.98$, $\eta=0.01$, $c_0=1.15$, see \cite{Pradas-Tseluiko-Kalliadasis-2011}. 

\begin{remark}
	 If the coefficients $\phi_0$ and $\phi_1$ in the left upper corner of $T_0$ \vspace{-1mm}satisfy 
	\begin{equation}\label{cond:omega1,2}
	\omega_1:=-\sup_{x\in\RR} \phi_0(x)>0, \quad \omega_2:=\sup_{x\in\RR}|\phi_1(x)|<\sqrt{\frac{9\eta}{\delta}\omega_1},
	\vspace{-1mm}
	\end{equation}
	which holds e.g.\ if $ \phi_1 \equiv c_0-\frac{17}{21}$ and $\phi_0 \equiv -\frac{5}{2\delta}$ on $\RR$,  
	the exceptional set $\Lambda$ in \eqref{ess.spec.film.1} is empty, and can thus can be omitted, 
	\begin{align}
	\label{verylast}
	\Lambda=\sigma(\ov{T_{a}}) \cap \ii\RR = \emptyset.
	\end{align}
	To see this, note that for any $\psi\in\Dom(T_a)=\sS(\RR)$ and $\varepsilon>0$, integration by parts and Young's inequality easily yield
	\begin{equation*}
	\Re \langle T_a\psi,\psi\rangle \leq \Bigl(\varepsilon\omega_2-\frac{9\eta}{2\delta}\Bigr)\|\psi'\|^2_{L^2(\RR)} + \Bigl(\frac{\omega_2}{2\varepsilon}-\omega_1\Bigr)\|\psi\|^2_{L^2(\RR)}.
	\end{equation*}
	Due to Assumption \eqref{cond:omega1,2} we can choose $\varepsilon=\varepsilon_0$ such that $\frac{\omega_2}{2\omega_1}<\varepsilon_0<\frac{9\eta}{2\delta\omega_2}$ and \vspace{-1mm} thus 
	\[
	\sup_{\substack{\psi\in\sS(\RR) \\ \psi\neq0}} \frac{\Re \langle T_a\psi,\psi\rangle}{\|\psi\|^2_{L^2(\RR)}} \leq \frac{\omega_2}{2\varepsilon_0}-\omega_1<0. \vspace{-1mm}
	\]
	Since we already know that $\sigma(\ov{T_{a}})\subset \ov{W(\ov{T_{a}})}$, \eqref{verylast} follows. 
	
%
\end{remark}

\begin{figure}[h!]
	\centering
	\begin{subfigure}{\textwidth}
		\centering
		\includegraphics[width=0.74\textwidth,height=7.4cm,fbox]{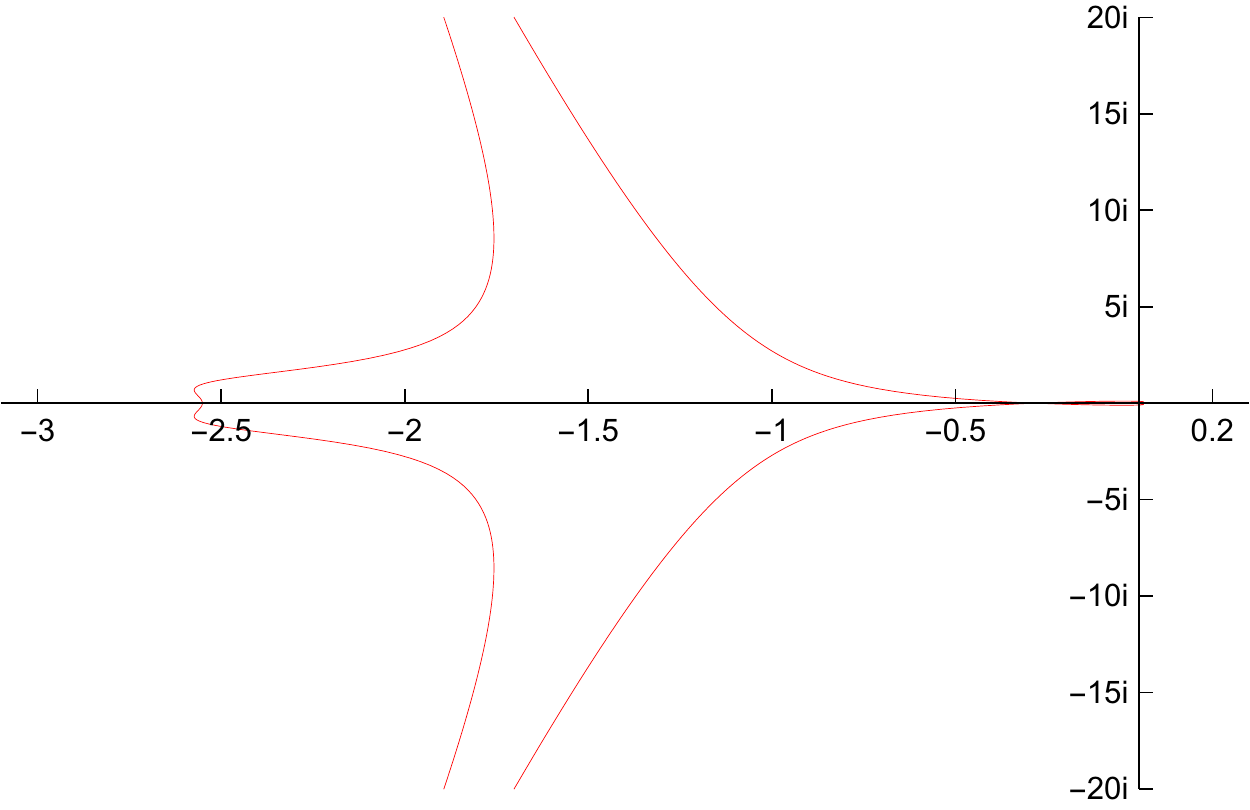}
		\caption{}	
	\end{subfigure}	\par\vspace{3mm}
		
	\begin{subfigure}{\textwidth}
		\centering
		\includegraphics[width=0.74\textwidth, height=7.4cm,fbox]{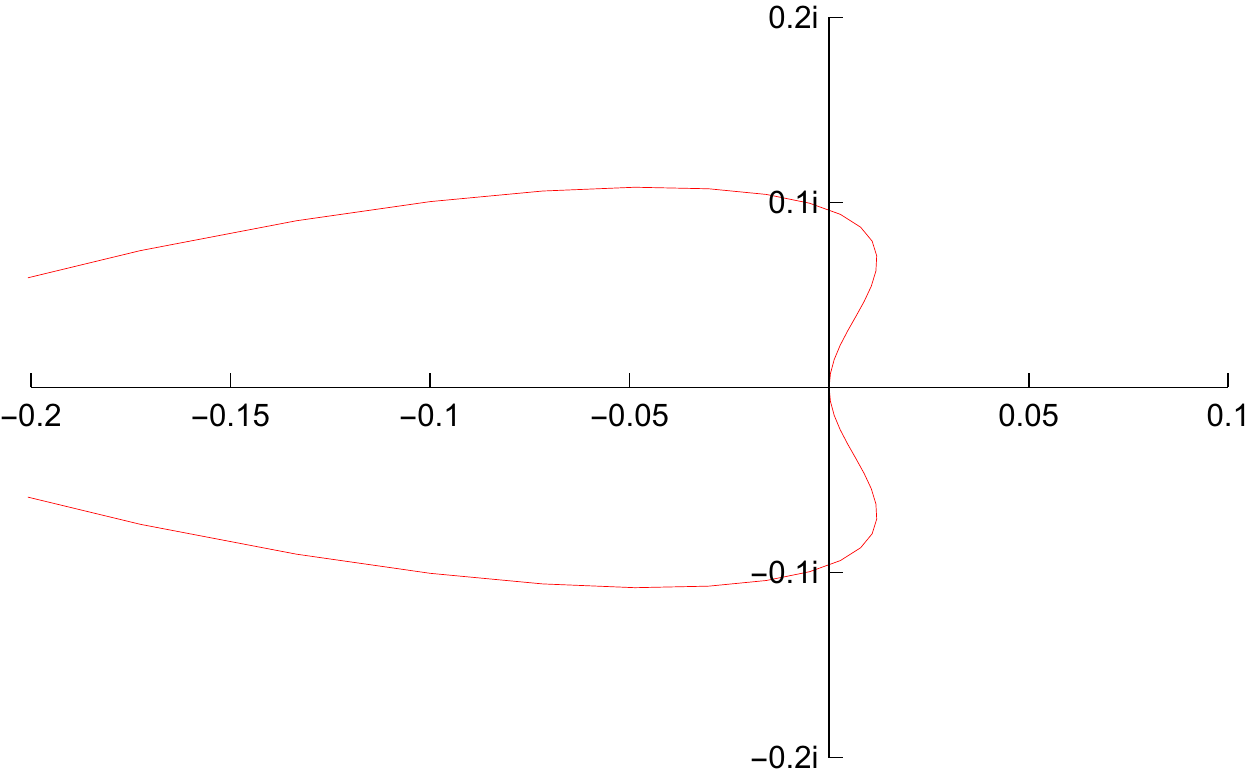}
		\caption{}
	\end{subfigure} 
	\caption{\small{The essential spectrum of the closure of the operator matrix $T_0$ in \eqref{T0Film} in $[-3, 0.2]\times[-20\ii,20\ii]$, see (a), and zoomed in near the imaginary axis in $[-0.2,0.1]\times[-0.2\ii,0.2\ii]$, see (b).}}\label{fig:film}
\end{figure}

\smallskip
\noindent
\textbf{Acknowledgements.} The authors gratefully acknowledge the support of the \emph{Swiss National Science Foundation,} SNF grant no.\ 
$200020\_146477$. The first author also gratefully acknowledges the support of SNF Early Postdoc.Mobility grant no.\ $168723$ and thanks the Department of Mathematics at University College London for the kind hospitality. 

{\small
	\bibliographystyle{acm}
	\bibliography{ref_pdoA}
}

\end{document}